\documentclass[a4paper]{article}
\pdfoutput=1
\usepackage{graphicx}
\usepackage{amssymb, amsmath, amsfonts,amsthm}
\usepackage[latin1]{inputenc}
\newcommand{\havf}{\ensuremath{H_{\mathrm{\scriptscriptstyle{AVF}}}}}
\newcommand{\chavf}{\ensuremath{\mathcal{H}_{\mathrm{\scriptscriptstyle{AVF}}}}}
\newcommand{\barp}{k}
\newcommand{\bw}{\ensuremath{\bf w}}

\newtheorem{theorem}{Theorem}[section]

\newtheorem{definition}[theorem]{Definition}

\numberwithin{equation}{section}
\numberwithin{table}{section}
\numberwithin{figure}{section}

\newcommand{\stc}[1]{}

\title{A general framework for deriving integral preserving numerical methods for PDEs}
\author{Morten Dahlby \and Brynjulf Owren}

\begin{document}
\maketitle
\begin{abstract}
	A general procedure for constructing conservative numerical integrators for time dependent partial differential equations is presented. In particular, 
	linearly implicit methods preserving a time discretised version of the invariant is developed for systems of partial differential equations with polynomial nonlinearities. The framework is rather general and allows for an arbitrary number of dependent and independent variables with derivatives of any order. It is proved formally that second order convergence is obtained. The procedure is applied to a test case and numerical experiments are provided.
\end{abstract}




\section{Introduction}
	Schemes that conserve geometric structure have been shown to be useful when studying the long time behaviour of dynamical systems.
	Such schemes are sometimes called geometric or structure preserving integrators \cite{MR2221614,MR2132573}.
	In this paper we shall mostly be concerned with the conservation of first integrals.

 Even if a presumption in this work is that the development of new and better integral preserving schemes is useful, we would still like to mention some situations where schemes with such properties are of importance. In the literature one finds several examples where stability of a numerical method is proved by directly using its conservative property, one example is the scheme developed for the cubic Schr\"{o}dinger equation in \cite{fei95nso}. Another application where the exact preservation of first integrals plays an important role is in the study of orbital stability of soliton solutions to certain Hamiltonian partial differential equations (PDEs) as discussed by Benjamin and coauthors \cite{benjamin72tso,benjamin72mef}.	
	
For ordinary differential equations (ODEs) it is common to devise relatively general frameworks for structure preservation.
 This is somewhat to the contrary of the usual practice with partial differential equations where each equation under consideration normally requires a dedicated scheme. 
But there exist certain fairly general methodologies that can be used for
developing geometric schemes also for PDEs. For example, through space discretisation of a Hamiltonian PDE one may obtain a system of Hamiltonian ODEs to which a geometric integrator may be applied. Another approach is to formulate the PDE in multi-symplectic form, and then apply a scheme which preserves a discrete version of this form, see \cite{MR2220764} for a review of this approach.

In this paper we consider methods for PDEs that are based on the discrete gradient method for ODEs. The discrete gradient method was perhaps first treated in a systematic way by Gonzalez \cite{MR1411343}, see also
\cite{MR2221614,MR1694701}. 
For PDEs one may derive discrete gradients either for the abstract Cauchy problem, where the solution at any time is considered as an element of some infinite dimensional  space, or one may semidiscretise the equations in space and then derive the corresponding discrete gradient for the resulting ODE system.
This last procedure has been elegantly presented in several articles by
Furihata, Matsuo and collaborators, see e.g.\ \cite{MR1727636,MR1852556,MR1815731,MR2313820,MR1848726,
MR1795452,MR1933890}, using the concept of discrete variational derivatives. See also the monograph \cite{fumabook}. The first part of this paper develops a similar framework that is rather general and allows for an arbitrary number of dependent and independent variables with derivatives of any order. The suggested approach does not require the equations to be discretised in space.

We consider a class of conservative schemes which are linearly implicit.
By linearly implicit we mean schemes which require the solution of precisely one linear system of equations in each time step. This is opposed to fully implicit schemes for which one typically applies an iterative solver that may require a linear system to be solved in every iteration. 
For standard fully implicit schemes one would typically balance the iteration error in solving the nonlinear system with the local truncation error. However, for conservative schemes the situation is different since exact conservation of the invariant requires that the nonlinear system is solved to machine precision.  This work can be seen as a generalisation  of ideas introduced in \cite[Section 6]{MR1848726}.

It may not, in general, be an easy task to quantify exactly what can be expected of gain in computational cost, if any, when replacing a fully implicit scheme with a linearly implicit one. For illustration we present an example where the KdV equation 
\begin{equation} \label{eq:KdV}
    u_t + u_{xxx}+(u^2)_x = 0
\end{equation}
is solved on a periodic domain using a fully implicit scheme
\begin{equation}\label{eq:kdvimp}
 \frac{U^{n+1}-U^{n}}{\Delta t}+\frac{U_{xxx}^{n+1}+U_{xxx}^{n}}2+\left(\frac{(U^{n+1})^2+U^{n+1}U^{n}+(U^{n})^2}3\right)_x=0,
\end{equation}
	and a linearly implicit scheme
\begin{equation}\label{eq:kdvlinin}
 \frac{U^{n+2}-U^{n}}{2\Delta t}+\frac{U_{xxx}^{n+2}+U_{xxx}^{n}}2+\left(U^{n+1}\frac{U^{n+2}+U^{n+1}+U^{n}}3\right)_x=0,
\end{equation}	
where $U^n(x)\approx u(x,t^n)=u(x,t^0+n\Delta t)$.
	
	These schemes are derived in Section \ref{se:num}.
	For space discretisation centered differences are used for both schemes.
	Note that the linearly implicit scheme \eqref{eq:kdvlinin} has a multistep nature, but should not be confused with standard linear multistep methods. Furihata et al.\  sometimes use the term "multiple points linearly implicit schemes" to emphasise this fact.

	The schemes are both second order in time, but in our example the linearly implicit multistep scheme has an error constant which is about 3-4 times larger than the fully implicit one-step scheme. 
	In Figure~\ref{fi:secondpic} we plot the global error versus the number of linear solves for the two 
	schemes. The linearly implicit scheme solves one linear system in each time step. 
	The fully implicit scheme, on the other hand, solves a linear system for each Newton iteration which is repeated to machine precision in each time step. 
	For the largest time step in our experiment this amounts to 561 linear solves per time step. The linear systems in each of the two cases 
	have the same matrix structure, they are both penta-diagonal, and we therefore assume that the cost of solving the linear system is approximately the same for both methods.
	The $x$-axis in Figure~\ref{fi:secondpic} can thus be interpreted as  a measure of the computational cost in each scheme. 
	The plot shows that for a given global error the linearly implicit scheme is computationally cheaper than the fully implicit scheme. 
	\begin{figure}
		\centering
		\includegraphics[width=.8\textwidth]{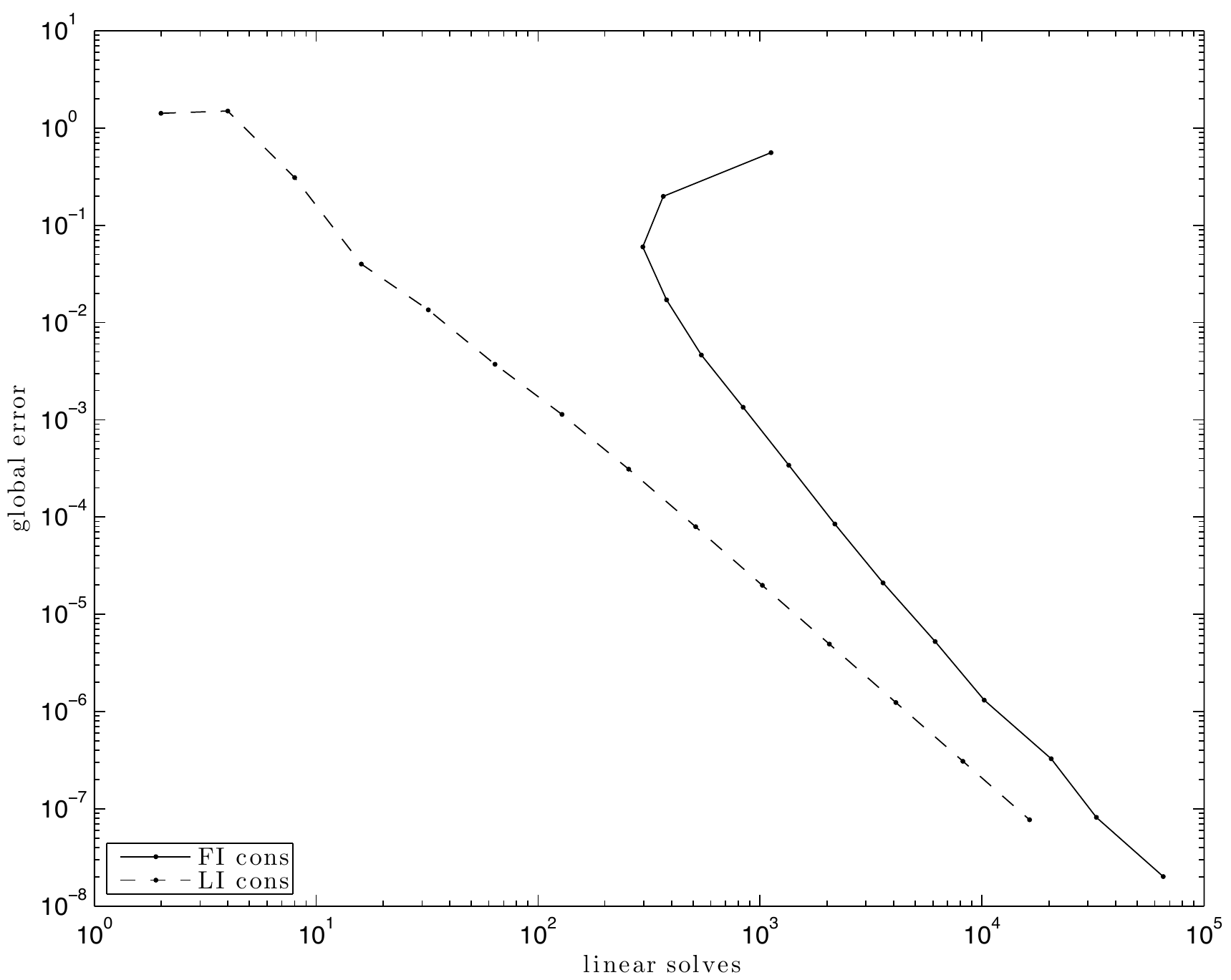}
		\caption{The global error versus the number of linear solves for the two schemes~\eqref{eq:kdvimp} (FI cons) and~\eqref{eq:kdvlinin} (LI cons).}
		\label{fi:secondpic}
	\end{figure}	
	
	There are situations in which the results from this example may be less relevant. For instance, the iteration method used in a fully implicit schemes may use approximate versions of the Jacobian for which faster solvers can be applied, therefore the cost of a linear solve may not be the same for the two types of schemes. For large time steps, both types of schemes are likely to encounter difficulties, but for slightly different reasons. The fully implicit scheme may experience slow or no convergence at all of the iteration scheme, whereas the linearly implicit scheme may become unstable for time steps over a certain threshold \cite{dahlbyowren}. For instance, in the case of the stiff ODE considered by Gonzales and Simo \cite{gonzalez96ots} one will observe that the stability properties of the linearly implicit schemes will be completely lost whereas fully implicit conservative schemes behave remarkably well.
	For large-scale problems, one may have situations in which iterative linear solvers are required and where one cannot afford to solve these systems to machine accuracy, in such cases the linearly implicit schemes are less useful.
	In conlusion, we believe that which of the two types of schemes that is preferable depends on the PDE and the circumstances under which it is to be solved. 
 
	The two schemes used in the example above have slightly different conservation properties. 	
			The first one \eqref{eq:kdvimp} conserves the exact Hamiltonian
	\begin{equation*}
		\mathcal{H}[U^n]=\int_\Omega\left(\frac12(U^n_x)^2-\frac13(U^n)^3\right)\,\mathrm{d}x,
	\end{equation*}
	whereas the second scheme \eqref{eq:kdvlinin} conserves what we will define as the polarised Hamiltonian
	\begin{equation*}
		H[U^n,U^{n+1}]=\int_\Omega\left(\frac14\left((U^n_x)^2+(U^{n+1}_x)^2\right)-\frac16\left((U^n)^2U^{n+1}+(U^{n+1})^2U^n\right)\right)\,\mathrm{d}x.
	\end{equation*}
	Both of these functions are approximations to the true Hamiltonian, the first is a spatial approximation for a fixed time, and the second also includes an averaging over time. 
	The intention is that in both cases one can see the methods as exactly preserving a slightly perturbed first integral over very long times.
	 That this seems to work for the chosen example is clearly seen in the first plot in Figure \ref{fi:Hplot} where we plot the error in $\mathcal{H}$ as a function of the solution obtained by the linearly implicit scheme \eqref{eq:kdvlinin}. We integrate to $t=1000$ and the error is plotted from $t=980$ to $t=1000$. 
	Notice that in this example there is no drift in the energy error. 			
	The corresponding error plots for $\mathcal{H}$ as a function of the solution of \eqref{eq:kdvimp} and $H$ as a function of the solution of \eqref{eq:kdvlinin} are omitted since they are both preserved up to round-off error, and thus not that interesting. 
	The second plot in Figure \ref{fi:Hplot} shows how the error in $\mathcal{H}$ at the endpoint depends on $\Delta t$. 
	Empirically, we have the relation
	\begin{equation*}
		\mathcal{H}[U^n]=\mathcal{H}[U^0]+C(\Delta t)^2,
	\end{equation*}
	where $C$ is a constant that depends on the solution, but not on $n$. 
	See Section \ref{se:num} for another example that tests the long time structure preserving properties of these schemes.


	\begin{figure}
		\centering
		\includegraphics[width=.8\textwidth]{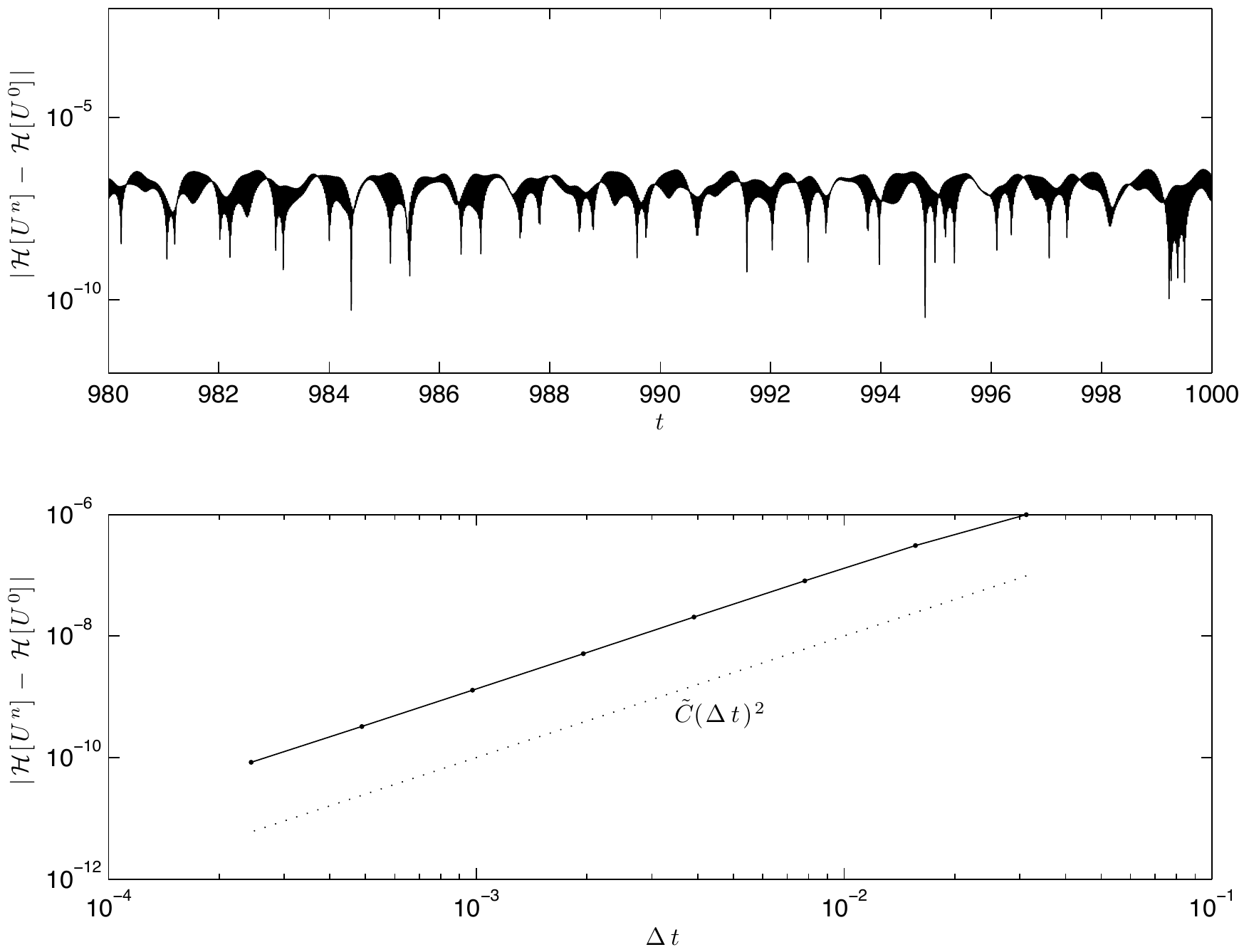}
		\caption{The error in $\mathcal{H}$ as a function of the solution obtained by the linearly implicit scheme \ref{eq:kdvlinin} as a function of $t$ (top) and $\Delta t$ (bottom). The dotted line is a reference line $\tilde{C}(\Delta t)^2$.}
		\label{fi:Hplot}
	\end{figure}
This and similar examples show that there are situations where linearly implicit schemes can be a better choice than their fully implicit counterparts. Figure \ref{fi:secondpic} shows that the linearly implicit scheme is cheaper, while Figure \ref{fi:Hplot} shows that both solutions have similar long-term behaviour. Similar favourable behaviour of linearly implicit schemes can be found in the literature.
		For the cubic Schr\"{o}dinger equation there are such conservative schemes based on time averaged versions of the Hamiltonian by Fei et al.\ \cite{fei95nso} and by Besse \cite{besse04ars}. Examples of methods for other PDEs can be found in the monograph \cite{fumabook} and the papers \cite{MR1360462} and \cite{MR1239931}. 
		
In the next section we define the PDE framework that we use. Then, in Section~\ref{sec:dg}
	we consider discrete gradient methods and how they can be applied to PDEs. We study in particular
	 the average vector field method by Quispel and McLaren \cite{MR2451073} and the discrete variational derivative method by
	 Furihata, Matsuo and coauthors
	\cite{MR1727636,MR1852556,MR2313820,MR1848726,MR1795452,MR1933890}.  We develop a framework that works for a rather general class of equations. 
	
	The key tools for developing linearly implicit methods for polynomial Hamiltonians are treated in Section~\ref{sec:linimp}, introducing the concept of polarisation. There is some freedom in this procedure, and we show through a rather general example term how the choice may significantly affect the stability of the scheme.
	
We defer the introduction of spatial discretisation until Section~\ref{se:spatial}. This is done mostly in order to keep a simpler notation, but also because our approach concerns conservative time discretisations and is essentially independent of the choice of spatial discretisation.  
The last section offers some more details on the procedure for constructing schemes and we give some indication through numerical tests on the long term behaviour of the schemes.
%
%
 \section{Notation and preliminaries}
 We consider integral preserving PDEs written in the form	
\begin{equation}
		u_t=\mathcal{D}\frac{\delta\mathcal{H}}{\delta u},
		\label{eq:hamileq}
	\end{equation}	
	where 
	\begin{equation}
		\mathcal{H}[u]=\int_{\Omega} \mathcal{G}[u]\,\mathrm{d} x=
		\int_{\Omega} \mathcal{G}((u_J^{\alpha}))\,\mathrm{d} x,\quad \Omega\subseteq\mathbb{R}^d, 
		 		\label{eq:h}
	\end{equation}	
	is the preserved quantity  and $\mathcal{D}$ is a skew-symmetric operator that may depend on $u$.  We write $\mathrm{d}x=\mathrm{d}x^1\cdots\mathrm{d}x^d$. We remark in passing that the class of PDEs which can be written in the form \eqref{eq:hamileq} contains the class of Hamiltonian PDEs, however we do not make the additional assumption that $\mathcal{D}$ satisfies the Jacobi identity \cite{olver93aol}. 
	By $(u_J^\alpha)$ we mean $u$ itself, which may be a vector 
$u=(u^{\alpha})\in\mathbb{R}^m$, and all its partial derivatives with respect to all independent variables, $(x^1,\ldots,x^d)$, up to and including some degree $\nu$. Thus, $J$ is a multi-index, we let $J=(j_1,\ldots,j_{r})$, where $r=|J|$ the number of components in $J$, and 
\[
   u_J^{\alpha} = \frac{\partial^{r}u^{\alpha}}{\partial x^{j_1}\cdots\partial x^{j_{r}}},\quad 0\leq r\leq\nu.
\]
As in \cite{olver93aol}, the square brackets in \eqref{eq:h} are used to indicate that a function depends also on the derivatives of its arguments with respect to the independent variables.
  In one dimension $d=1$ and $m=1$, for example, we have
	\begin{equation*}
		\mathcal{G}[u]=\mathcal{G}\left((u_J)\right)=
		\mathcal{G}\left(u,\frac{\partial u}{\partial x},\dots,
		\frac{\partial^\nu u}{\partial x^\nu}\right).
	\end{equation*}
The variational derivative $\frac{\delta\mathcal{H}}{\delta u}$ is an $m$-vector depending on
$u_J^{\alpha}$ for   $|J|\leq \nu'$ where $\nu'\geq \nu$. It may be defined through the relation
	\cite[p. 245]{olver93aol}
	\begin{equation}\label{eq:varder}
		\int_{\Omega} \frac{\delta\mathcal{H}}{\delta u}\cdot \varphi\,\mathrm{d}x
		=\left.\frac{\partial}{\partial \epsilon}\right|_{\epsilon=0}\mathcal{H}[u+\epsilon \varphi],
	\end{equation}
	for any sufficiently smooth $m$-vector of functions $\varphi(x)$. One may calculate
 $\frac{\delta\mathcal{H}}{\delta u}$  by applying the Euler operator to $\mathcal{G}[u]$,
 the $\alpha$-component is given as
 \begin{equation} \label{eq:varderdef}
\left(\frac{\delta\mathcal{H}}{\delta u} \right)^{\alpha} =   \mathbf{E}_{\alpha}\mathcal{G}[u],
 \end{equation}
 where
 \begin{equation} \label{eq:euleropdef}
     \mathbf{E}_{\alpha} = \sum_{|J|\leq\nu} (-1)^{|J|} D_J \frac{\partial}{\partial u_J^{\alpha}}
 \end{equation}
 so that the sum ranges over all $J$ corresponding to derivatives $u_J^{\alpha}$ featuring in
 $\mathcal{G}$.
 We have used total derivative operators,
 \[
     D_J=D_{j_1}\dots D_{j_k},\quad D_{i}=\sum_{\alpha,J} \frac{\partial u_J^{\alpha}}{\partial x^i}
     \frac{\partial}{\partial u_J^{\alpha}}.
 \]
 	
	 In parts of the paper we refer to Hamiltonians as polynomial, or specifically quadratic. By this we mean that $\mathcal{H}$ is of a form such that $\mathcal{G}$ is a multivariate polynomial in the indeterminates $u_J^{\alpha}$, which in the quadratic case is of degree at most two. For example, the KdV equation \eqref{eq:KdV} has a polynomial Hamiltonian of degree 3
	\begin{equation*}
		\mathcal{H}[u]=\int_\Omega\left(\frac12u_x^2-\frac{1}{3}u^{3}\right)\,\mathrm{d}x.
	\end{equation*} 
	 In this case $\mathcal{G}=\mathcal{G}(u,u_{x})$ and thus $m=d=\nu=1$, and we get	
	 \begin{align}
		\frac{\delta\mathcal{H}}{\delta u}=\mathbf{E}\mathcal{G}((u_J))&=
		\frac{\partial\mathcal{G} }{\partial u}-\frac{\partial}{\partial x}\frac{\partial\mathcal{G} }{\partial u_{x}}\label{eq:euler111}\\
		&=-u^2-u_{xx}.
	\end{align}

	We always assume sufficient regularity in the solution and that the boundary conditions on $\Omega$ are such that the boundary terms vanish when doing  integration by parts, for example periodic boundary conditions. 
 	The operator $\mathcal{D}$ should be skew-symmetric with respect to the $L^2$ inner product
	\begin{equation}
		\int_{\Omega} (\mathcal{D}v)w\,\mathrm{d}x=-\int_{\Omega} v(\mathcal{D}w)\,\mathrm{d}x\quad
		\forall\; u,w.
		\label{eq:skew}
	\end{equation}
	For the KdV case we simply have $\mathcal{D}=\frac{\partial}{\partial x}$.
	
Furthermore, to be a true Hamiltonian system it should induce a Poisson bracket on the space of functionals as described e.g.\ in \cite[Ch 7.1]{olver93aol}, meaning that the Jacobi identity must be satisfied. However, the approach presented here only requires $\mathcal{D}$ to be skew-symmetric so that the functional $\mathcal{H}$ is a conserved quantity. In the case that the PDE has more than one Hamiltonian formulation, we may make a choice of which of the integrals to preserve. Our approach does not in general allow for the preservation of more than one Hamiltonian at the same time, for this see the upcoming paper \cite{daowya2010}.

PDEs such as the wave equation are typically written with $u_{tt}$ appearing on the left hand side, in such cases we double the dimension of $u$ in order to apply the stated framework.
For complex equations one may do something similar, splitting either into a real and an imaginary part, or adding in the complex conjugate as a separate variable.
 	
\section{Discrete gradient and variational derivative methods}\label{sec:dg}
Discrete gradient methods for ODEs were introduced by Gonzalez  \cite{MR1411343}. 
See also \cite{cell092}, \cite{cell093},  \cite{MR1694701}, and \cite[Chapter V.5]{MR2221614}. 
Recently  this idea has been applied to PDEs in the form of 
the  average vector field (AVF) method \cite{cell09} and in a somewhat more general setting, the discrete variational derivative (DVD) method. 

We recall the definition of a discrete gradient as presented for ODEs. If $H:\mathbb{R}^M \rightarrow\mathbb{R}$, a discrete gradient is a continuous map $\overline{\nabla}:
\mathbb{R}^M\times \mathbb{R}^M\rightarrow\mathbb{R}^M$ such that for every $\bf u$ and
$\bf v$ in $\mathbb{R}^M$
\begin{align*}
H({\bf u})-H({\bf v})&= \overline{\nabla}H({\bf v},{\bf u})\cdot ({\bf u}-{\bf v}),\\
\overline{\nabla}H({\bf u},{\bf u}) &= \nabla H({\bf u}).
\end{align*}
Since an ODE system preserving $H$ can be written in the form
\[
   \frac{\mathrm{d}\bf  y}{\mathrm{d}t} = S({\bf y})\,\nabla H({\bf y})
\]
for some skew-symmetric matrix $S({\bf y})$, one obtains a conservative method simply by defining approximations ${\bf y}^n \approx {\bf y}(t^n)={\bf y}(t^0+n\Delta t)$ through the formula
\[
     \frac{{\bf y}^{n+1}-{\bf y}^n}{\Delta t} = \tilde{S}\,\overline{\nabla}H({\bf y}^n,{\bf y}^{n+1}),
\]
where $\tilde{S}$, typically allowed to depend on ${\bf y}^n$ and ${\bf y}^{n+1}$, is some skew-symmetric matrix approximating the original $S$.

There are many possible choices of discrete gradients for a function $H$, see for instance
\cite{MR2221614,MR1694701}.
A particular example  is the one used in the AVF method defined as
\[
\overline{\nabla}_{\mathrm{AVF}}H({\bf v},{\bf u}) = \int_0^1 \nabla H(\xi {\bf u}+(1-\xi){\bf v})\,\mathrm{d}\xi.
\]
 When applying this approach to PDEs the obvious strategy is to discretise the Hamiltonian
$\mathcal{H}[u]$ in space, replacing each derivative by a suitable approximation like e.g.\ finite differences, 
 to obtain a Hamiltonian $\mathcal{H}_d({\bf u})$ as for ODEs. Similarly, the skew-symmetric operator $\mathcal{D}$ is replaced by a skew-symmetric $M\times M$-matrix $\mathcal{D}_d$ to yield the scheme
 \begin{equation} \label{eq:metdg}
      \frac{{\bf u}^{n+1}-{\bf u}^n}{\Delta t} = \mathcal{D}_d\,\overline{\nabla}\mathcal{H}_d({\bf u}^n,{\bf u}^{n+1})
 \end{equation}
 for advancing the numerical solution ${\bf u}^n$ at time $t^n$ to ${\bf u}^{n+1}$ at time $t^{n+1}$.
Examples are worked out for several PDEs in \cite{cell09}.

Furihata, Matsuo and coauthors present a whole framework for discretising PDEs in the variational setting in a series of papers, providing a discrete analogue of the continuous calculus, see for instance \cite{MR1727636}.
They discretise $\mathcal{G}$ to obtain $\mathcal{G}_d$ using difference operators, 
and then the integral in $\mathcal{H}$ is approximated by a sum to yield $\mathcal{H}_d$.
Then they derive a discrete counterpart to the variational derivative, and finally state the difference scheme in a form which is a perfect 
analogue to the Hamiltonian PDE system \eqref{eq:hamileq}, letting
\[
\frac{\mathbf{u}^{n+1}-\mathbf{u}^n}{\Delta t} = \mathcal{D}_d\frac{\delta\mathcal{H}_d}{\delta({\bf u}^n,{\bf u}^{n+1})}.
\]
 The use of integration by parts in deriving the Euler operator is mimicked by similar summation by part formulas for the discrete case.
Their discrete variational derivative is in fact rather similar to a discrete gradient, as it satisfies
the relation
\begin{equation} \label{eq:dvdprop}
 \mathcal{H}_d({\bf u})-\mathcal{H}_d({\bf v}) = 
 \langle\frac{\delta\mathcal{H}_d}{\delta({\bf v},{\bf u})},{\bf u}-{\bf v}\rangle
\end{equation}
for the discrete $L^2$ inner product. 

In the present paper, we focus on the time dimension in most of what follows, thus we shall defer the steps in which $\mathcal{H}$ and thereby
$\mathcal{G}$ are discretised in space. But \eqref{eq:dvdprop} makes perfect sense after removing the subscript $d$, replacing  $\mathbf{u}$ and $\mathbf{v}$ by functions $u$ and $v$, and the discrete $L^2$ inner product by the continuous one.
A discrete variational derivative (DVD) is here defined to be any continuous function  $\frac{\delta  \mathcal{H}}{\delta (v,u)}$ of $(u^{(\nu)},v^{(\nu)})$ satisfying
	\begin{align}
		\mathcal{H}[u]-\mathcal{H}[v]&=\int_\Omega \frac{\delta  \mathcal{H}}{\delta (v,u)}(u-v)\,\mathrm{d}x,\label{eq:dg1}\\
		\frac{\delta  \mathcal{H}}{\delta (u,u)}&=\frac{\delta\mathcal{H}}{\delta u}.\label{eq:dg2}
	\end{align}
The integrator yields a continuous function $U^n:=U^n(x)\approx u(x,t^n)$ for each $t^n$
	\begin{equation}\label{eq:dgm}
		\frac{U^{n+1}-U^{n}}{\Delta t}=\mathcal{D}\frac{\delta \mathcal{H}}{\delta (U^n,U^{n+1})}.
	\end{equation}
	 By combining \eqref{eq:dg1} and \eqref{eq:dgm} we see that the method preserves $\mathcal{H}$.
  	
The AVF scheme can of course also be interpreted as a discrete variational derivative method where 
		\begin{equation}\label{eq:dgavf}
		\frac{\delta \chavf }{\delta (v,u)}=\int_{0}^{1}\frac{\delta\mathcal{H}}{\delta u}[\xi u+(1-\xi)v]\,\mathrm{d}\xi.
	\end{equation}
The fact that \eqref{eq:dgavf} verifies the condition \eqref{eq:dg1} is seen from the elementary identity
\begin{equation} \label{eq:diffH}
\mathcal{H}[u]-\mathcal{H}[v] = \int_0^1 \frac{\rm d}{\mathrm{d}\xi}
\mathcal{H}[\xi u + (1-\xi) v]\,\mathrm{d}\xi.
\end{equation}
The derivative under the integral is written
\begin{align*}
 \frac{\rm d}{\mathrm{d}\xi}\mathcal{H}[\xi u + (1-\xi) v]&=
  \left.\frac{\rm d}{\mathrm{d}\varepsilon}\right|_{\varepsilon=0}\mathcal{H}[v+(\xi+\varepsilon) (u - v)]
    \\[2mm]
    & = \int_{\Omega} \frac{\delta \mathcal{H}}{\delta u}[\xi u + (1-\xi) v] (u-v)\,\mathrm{d}\mathbf{x}.
\end{align*}
Now substitute this into \eqref{eq:diffH} and interchange the integrals to obtain \eqref{eq:dg1}.

In most of the cited papers by Furihata, Matsuo and coauthors, the notion of a DVD method is less general than what we just presented, in the sense that  the relation \eqref{eq:dvdprop} is not actually used as the defining equation for a discrete variational derivative. Instead the authors present a relatively general format that can be used for discretising $\mathcal{H}$, this format is depending on the class of PDEs under consideration, and they work out the explicit expression for a particular discrete variational derivative. To give an idea of how the format may look like, we briefly review some points from \cite{MR1727636} where PDEs of the form \eqref{eq:h} are considered with $d=m=\nu=1$ such that $\mathcal{G}=\mathcal{G}(u,u_x)$. $\mathcal{G}$ is assumed to be written as a finite sum
\begin{equation} \label{eq:japform}
     \mathcal{G}(u,u_x) = \sum_{\ell} \alpha_{\ell} f_{\ell}(u)g_{\ell}(u_x).
\end{equation}
where $f_{\ell}$ and $g_{\ell}$ are differentiable functions.
\footnote{In \cite{MR1727636} the expression is discretised in space and $g_{\ell}(u_x)$ 
is replaced by a product $g_{\ell}^+(\delta_k^+U_k)g_{\ell}^-(\delta_k^-U_k)$ where
$\delta_k^+$ and $\delta_k^-$ are forward and backward divided differences respectively.}
The form \eqref{eq:dg1} is then derived through
\begin{align*}
 f_{\ell}(u)g_{\ell}(u_x)- f_{\ell}(v)g_{\ell}(v_x)
  &= \frac{f_\ell(u)-f_\ell(v)}{u-v}\frac{g_\ell(u_x)+g_\ell(v_x)}{2}(u-v)\\&
  + \frac{g_\ell(u_x)-g_\ell(v_x)}{u_x-v_x}\frac{f_\ell(u)+f_\ell(v)}{2}(u_x-v_x)
\end{align*}
followed by an integration by part on the second term. This technique can be extended in any number of ways to allow for more general classes of PDEs. For instance, one may allow for more factors in \eqref{eq:japform}, like 
\[
   \mathcal{G}[u]=\sum_{\ell}\alpha_{\ell}\prod_{J} g_{\ell,J}(\partial_J u)
\]
and repeated application of the formula $ab-cd=\frac{a+c}{2}(b-d)+\frac{b+d}{2}(a-c)$ to this equation combined with integration by parts will result in a discrete variational derivative.

Schemes which are built on this particular type of discrete variational derivative will be called the Furihata methods in the sequel since it was first introduced 
in~\cite{MR1727636}. Matsuo et al.\ extend the method 
to complex equations in \cite{MR1848726},	while~\cite{MR1852556,MR2313820}  derive methods for equations with second order time derivatives. 
	Other papers using the discrete variational derivative approach include \cite{MR1815731}, \cite{MR1933890}, and \cite{Yaguchi2009}.
	
The lack of a general formalism in the papers just mentioned, makes it somewhat difficult to compare the approach to the AVF method and characterise in which cases they lead to the same scheme. Taking for instance the KdV equation \eqref{eq:KdV} one easily finds that both approaches lead to the scheme
\eqref{eq:kdvimp}, however, considering for instance the Hamiltonian
\[
    \mathcal{H}[u] = \int_\Omega u u_{x}^2\,\mathrm{d}x
\]
one would obtain two different types of discrete variational derivative in the Furihata method and the AVF method, that is
$\frac{\delta \mathcal{H}_{F}}{\delta (v,u)}\neq\frac{\delta \chavf}{\delta (v,u)}$.
 In some important cases, the Furihata method and the AVF method lead to the same scheme.
	\begin{theorem}\label{thm:avfeqfur}
	Suppose that the Hamiltonian $\mathcal{H}[u]$ is a linear combination of terms of either of the types
		\begin{enumerate}
			\item 
			$\int_\Omega\partial_{J}u\,\cdot\,\partial_{K}u\,\mathrm{d}x$ for multi-indices $J$ and $K$, or
			\item  
			$\int_\Omega g(\partial_J u)\,\mathrm{d}x$ for differentiable $g: \mathbb{R}\rightarrow \mathbb{R}$.
		\end{enumerate}
		Then the AVF and the Furihata methods yield the same scheme
	\end{theorem}
	\begin{proof} It suffices to check one general term of each type.
		\begin{enumerate}
			\item We find the variational derivative using~\eqref{eq:varder}
				\begin{equation*}
					\frac{\delta\mathcal{H}}{\delta u}=\left((-1)^{|J|}+(-1)^{|K|}\right)\partial_{J+K}u.
				\end{equation*}
				Inserting the variational derivative into~\eqref{eq:dgavf} gives 				
				\begin{equation*}
					\frac{\delta  \chavf}{\delta (v,u)}=\left((-1)^{|J|}+(-1)^{|K|}\right)\partial_{J+K}\left(\frac{u+v}{2}\right).			
				\end{equation*}
				To find the discrete variational derivative of the Furihata method we compute
				 				\begin{align*}
					\mathcal{H}[u]-\mathcal{H}[v]&=\int_\Omega\partial_{J}u\cdot\partial_{K}u-\partial_{J}v\cdot\partial_{K}v\,\mathrm{d}x\\
					&=\frac12\int_\Omega\left(\partial_{J}u-\partial_{J}v\right)\cdot\left(\partial_{K}u+\partial_{K}v\right)
					+\left(\partial_{J}u+\partial_{J}v\right)\cdot\left(\partial_{K}u-\partial_{K}v\right)\,\mathrm{d}x.
				\end{align*}
				After integration by parts  we get 			
				\begin{equation*}
					\frac{\delta  \mathcal{H}_{F}}{\delta (v,u)}=\left((-1)^{|J|}+(-1)^{|K|}\right)\partial_{J+K}\left(\frac{u+v}{2}\right),		
				\end{equation*}	
				and we see that $\displaystyle{\frac{\delta  \chavf}{\delta (v,u)}=\frac{\delta  \mathcal{H}_{F}}{\delta (v,u)}}$.
			\item In this case we get 
			\begin{equation*}
			\frac{\delta\mathcal{H}}{\delta u}=(-1)^{|J|}\partial_{J} g'(\partial_{J}u),
			\end{equation*}
			so that
			\begin{align*}
			\frac{\delta \chavf}{\delta (v,u)}&=
			(-1)^{|J|}\int_0^1 \partial_{J}g'(\partial_{J}(\xi u+(1-\xi)v))\,\mathrm{d}\xi
			\\&=(-1)^{|J|}\partial_J\left(\frac{g(\partial_J u)-g(\partial_J v)}
			            {\partial_Ju - \partial_Jv}
			\right).
			\end{align*}
			For the Furihata method one would here just compute
			\begin{align*}
			H[u]-H[v]=\int_\Omega \frac{g(\partial_J u)-g(\partial_J v)}
			{\partial_J u-\partial_J v}\left(\partial_J u-\partial_Jv
			\right)\,\mathrm{d}x
			\end{align*}
			and integration by parts yields $\displaystyle{\frac{\delta \mathcal{H}_{F}}{\delta (v,u)}=\frac{\delta \chavf}{\delta (v,u)}}$.	
 		\end{enumerate}
	\end{proof}	
	
\section{Linearly Implicit Difference Schemes} \label{sec:linimp}
	
\subsection{Polarisation}\label{se:pol}

		The key to constructing conservative linearly implicit schemes will be to portion out the nonlinearity over consecutive time steps. In effect, this means that we replace the original Hamiltonian $\mathcal{H}$ with an approximate one $H$. We shall call $H$  a polarisation 
of $\mathcal{H}$ since its definition resembles the way an inner product is derived from a quadratic form. We shall see that the difference scheme resulting from such a polarised Hamiltonian will be a multistep method. This method will now preserve exactly $H$, as opposed to $\mathcal{H}$ for the methods in the previous section. 
%
		The requirements on $H$ are given in the following definition.
		\begin{definition}[The polarised Hamiltonian]\label{de:polar}
			Given a Hamiltonian $\mathcal{H}[u]$ the polarised Hamiltonian $H$ depends on $\barp$ arguments, and
			is:
			\begin{itemize}
				\item Consistent 
					\begin{equation} \label{eq:consist}
						H[u,u,\dots,u]=\mathcal{H}[u].
					\end{equation}
				\item Invariant under any cyclic permutation of the arguments
					\begin{equation} \label{eq:cyclic}
					H[w_1,w_2,\dots,w_{\barp}]=H[w_2,\dots,w_{\barp},w_1].
					\end{equation}	
			\end{itemize}
		\end{definition}	
		Polarisations exist for any Hamiltonian, this is asserted by the example
		\begin{equation*}
			H[w_1,w_2,\dots,w_{\barp}] = \frac{1}{\barp}\left(\mathcal{H}[w_1]+\mathcal{H}[w_2]+\cdots+\mathcal{H}[w_{\barp}]\right).
		\end{equation*}	
We may impose the polarisation directly on the density $\mathcal{G}((u_J^{\alpha}))$, letting
\[
    H[w_1,w_2,\ldots,w_{\barp}] = \int_\Omega G[w_1,w_2,\ldots,w_{\barp}]\,\mathrm{d}x.
\]
The conditions \eqref{eq:consist}, \eqref{eq:cyclic} are then inherited as
\[
G(u,u,\ldots,u)=\mathcal{G}(u),\qquad G[w_1,w_2,\dots,w_{\barp}]=G[w_2,\dots,w_{\barp},w_1].
\]
In Section \ref{se:lin} we will discuss local order of consistency, it will then be convenient to make the stronger assumption that $\mathcal{H}$ and $H$ are at least twice Fr\'{e}chet differentiable. 
To distinguish from the weaker notion of variational (G\^{a}teaux) derivative, we replace
$\delta$ by $\partial$, noting that the first derivative in the two definitions are the same when they both exist.
We then find from \eqref{eq:consist} and \eqref{eq:cyclic} that the Fr\'{e}chet derivatives satisfy the relation
\begin{equation} \label{eq:firstder}
\frac{\partial\mathcal{H}}{\partial u}[u] = \barp\,\frac{\partial H}{\partial w_1}[u,\dots,u].
\end{equation}
For the second derivatives, we find the identity
\begin{equation} \label{eq:2dpolar}
\frac{\partial^2 H}{\partial w_1\partial w_j}
[u,\ldots,u] = \frac{\partial^2 H}{\partial w_1\partial w_{\barp+2-j}}[u,\ldots,u],\quad
j=2,\ldots, \lfloor \barp/2\rfloor + 1,
\end{equation}
which is used to compute
\begin{equation}  \label{eq:secder}
\frac{\partial^2 \cal H}{\partial u^2}[u]=
\left\{
\begin{array}{ll}
\displaystyle{\barp\left(\frac{\partial^2 H}{\partial w_1^2}
+2\sum_{\ell=2}^{\frac{\barp+1}{2}}\frac{\partial^2 H}{\partial w_1\partial w_\ell}\right)},
&\barp\ \mbox{odd,} \\
\displaystyle{\barp\left(\frac{\partial^2 H}{\partial w_1^2}
+2\sum_{\ell=2}^{\frac{\barp}{2}}\frac{\partial^2 H}{\partial w_1\partial w_\ell}
+ \frac{\partial^2H}{\partial w_1\partial w_{\frac{\barp}{2}+1}} \right)},
&\barp\ \mbox{even,}
\end{array}
\right.
\end{equation}
all second derivatives on the right being evaluated at $[u,\dots,u]$.

\subsubsection{Polynomial Hamiltonians}	
		The polarisation of polynomial Hamiltonians will be key to constructing linearly implicit schemes. 
		We will now explain in detail how to do this, and we begin with an example term in the 
		integrand $\mathcal{G}[u]=\mathcal{G}(\partial_J^{\alpha} u)$ depending on just one scalar indeterminate, namely $\mathcal{G}(z)=z^p$ where $z=\partial_J u^\alpha$ for some $(J, \alpha)$
		and where $p\leq 4$. 
		This example is important not only as a simple illustration of the procedure, but also because terms of this type are common in many of the Hamiltonians found in physics. As we will see in the next section, it will be natural to use two arguments, $\barp=2$, in the polarised Hamiltonian. In fact, we need to restrict ourself to cases with polynomial Hamiltonians for our technique to yield linearly implicit schemes. Then, by using $\barp \geq\lceil p/2\rceil$, we can obtain polarised Hamiltonians which are at most quadratic in each argument. We call these quadratic polarisations. We see that if $\barp=2$ then cyclic is the same as symmetric $G(u,v)=G(v,u)$, and
 the possible quadratic polarisations for $p=2,3,4$ are respectively, 
		\begin{align}
	p&=2:&			G(u,v)&=\theta\frac{{u}^2+{v}^2}2+(1-\theta){u}{v} \,\label{eq:hh1},\quad\theta\in[0,1], \\
	p&=3:&			G(u,v)&={u}{v}\frac{{u}+{v}}{2},\label{eq:hh2}\\
	p&=4:&			G(u,v)&={u}^2{v}^2.\label{eq:hh3}
		\end{align}
		Note that for these monomials
		both the third and fourth degree case are uniquely given, but the second degree case is not. In Section~\ref{se:stability}
		we will consider how the choice of $\theta$ influences the stability of the scheme for a term which appears frequently in PDEs.
		
We now consider the general case when $\mathcal{G}[u]$ is a multivariate polynomial in $N_\nu$ variables of degree $p$. It suffices in fact to let $\mathcal{G}((u_J^{\alpha}))$ be a monomial since each term can be treated separately, for $u\in\mathbb{R}^{N_\nu}$. For a convenient notation, we rename the vector of indeterminates $(u_J^\alpha)$ by using a single index i.e. $u=(u_1,\ldots,u_{N_\nu})$ and write 
\begin{equation} \label{eq:polar}
   \mathcal{G}(u) = u_{i_1}u_{i_2}\dots u_{i_p}.
\end{equation}
One may use the following procedure for obtaining a quadratic polarisation from \eqref{eq:polar}
\begin{enumerate}
\item Group the factors of the right hand side of \eqref{eq:polar} into pairs
$z_r = u_{i_{2r-1}}u_{i_{2r}}$ and if $p$ is odd $z_{\barp}=u_{i_p}$. Set
\[
  K(z_1,\ldots,z_{\barp})=z_1\cdots z_{\barp}.
\]
Note that there are potentially many ways of ordering the factors in \eqref{eq:polar} which give rise to different polarisations.
\item Symmetrise $K$ with respect to the cyclic subgroup of permutations. Letting the left shift permutation $\sigma$ be defined through
$\sigma K(z_1,\ldots,z_{\barp}) = K(z_2,\ldots,z_{\barp},z_1)$, we set
\[
    G(z_1,\ldots,z_{\barp})=\frac{1}{\barp}\sum_{k=1}^{\barp} \sigma^{k-1}K(z_1,\ldots,z_{\barp}).
\]
The resulting $G$ is now both consistent \eqref{eq:consist} and cyclic \eqref{eq:cyclic}.
\end{enumerate}

\subsection{Linearly Implicit Methods}\label{se:lin}
We may now define the discrete variational derivative for this polarised Hamiltonian as a generalisation of \eqref{eq:dg1} and \eqref{eq:dg2}. We let 
\[
   \frac{\delta H}{\delta(w_1,\dots,w_{\barp + 1})}
\]
be a continuous function of $\barp+1$ arguments, satisfying
\begin{equation} \label{eq:mdg1}
H[w_2,\dots,w_{\barp +1}]-H[w_1,\dots,w_{\barp}] =
\int_\Omega \frac{\delta H}{\delta(w_1,\dots,w_{\barp + 1})}(w_{\barp +1}-w_1)\,\mathrm{d}x,
\end{equation}
\begin{equation} \label{eq:mdg2}
\barp\,\frac{\delta H}{\delta(u,\dots,u)} = \frac{\delta\mathcal{H}}{\delta u}.
\end{equation}
Our standard example will be a generalisation of the AVF discrete variational derivative, which we define as
\begin{equation}\label{eq:tja}
\frac{\delta \havf }{\delta(w_1,\ldots,w_{\barp +1})}=
\int_0^1 \frac{\delta H}{\delta w_1}[\xi w_{\barp + 1}+(1-\xi) w_1,w_2,\ldots,w_{\barp}]
\,\mathrm{d}\xi.
\end{equation}
Here the variational derivative on the right hand side, $\frac{\delta H}{\delta w_1}$ is defined as before, considering $H$ as a function of 
its first argument only, leaving the others fixed.
Similar discrete variational derivatives could be derived in a number of different ways. In particular one finds that when the function $H$ is quadratic in all its arguments, the approach used in deriving the Furihata methods would lead to a discrete variational derivative which is identical to that of the AVF-method. 

Now we define the polarised discrete variational derivative scheme and prove that, under some assumptions, this scheme is conservative, linearly implicit and has formal order of consistency two.

\begin{definition}
For a Hamiltonian PDE of the form \eqref{eq:hamileq}, let $H$ be a polarised Hamiltonian
of $\barp$ arguments, satisfying \eqref{eq:consist} and \eqref{eq:cyclic}, and suppose that approximations
$U^j$ to $u(j\Delta t,\cdot)$ are given for $j=0,\ldots,\barp-1$.
\begin{itemize}
\item
	The polarised DVD (PDVD) scheme is given as 
	\begin{equation}\label{eq:scheme}				
		\frac{U^{n+\barp}-U^{n}}{\barp\Delta t}=\barp
		D\frac{\delta H}{\delta(U^{n},\dots,U^{n+\barp})},\quad n\geq 0.
	\end{equation}
\item
 $D$ is a skew-symmetric operator approximating $\mathcal{D}$.
 In \eqref{eq:scheme}, $D$ may depend on $U^{n+j}$, $1\leq j \leq  \barp-1$
 \footnote{$D$ should not depend on $U^{n+k}$ since otherwise the method would no longer be linearly implicit.},
  and be consistent
\begin{equation}\label{eq:Dconsistent}
D[u,\ldots,u]=\mathcal{D}[u].
\end{equation}
$D$ is called cyclic if
\begin{equation} \label{eq:Dcyclic}
D[w_1,w_2,\ldots,w_{k-1}]=D[w_{2},\ldots,w_{k-1},w_1].
\end{equation}
\item
If the discrete variational derivative is given by
\eqref{eq:tja}, then the scheme is called the polarised AVF  (PAVF) scheme. 	
\end{itemize}
\end{definition}	

	\begin{theorem}\label{th:cons}
	The scheme \eqref{eq:scheme}  is conservative
	in the sense that
 \begin{equation*}
      H[U^{n+1},\ldots,U^{n+k}]=H[U^0,\ldots,U^{k}],\quad\forall n\geq 1.
 \end{equation*}
 for any polarised Hamiltonian function $H$.
	\end{theorem}
\begin{proof}
By induction, this is an immediate consequence of \eqref{eq:mdg1}

\end{proof}	
In a framework as general as the one presented here, it is not possible to present a general analysis for convergence or the order of the truncation error.
However, it seems plausible that a necessary condition to obtain a prescribed order of convergence can be derived through a formal Taylor expansion of the local truncation error, we denote this \emph{the formal order of consistency}.

\begin{theorem}\label{th:order}\hspace*{0 cm}
\begin{itemize}	
\item	The PAVF scheme has formal order of consistency one
for any polarised Hamiltonian, and skew-symmetric operator $D$ satisfying \eqref{eq:Dconsistent}.
\item If in addition \eqref{eq:Dcyclic} is satisfied, the scheme has formal order of consistency two.
\end{itemize}
	\end{theorem}
\begin{proof}
We show that when the exact solution is substituted into \eqref{eq:scheme} where the
discrete variational derivative is given by \eqref{eq:tja}, then the residual is $\mathcal{O}(\Delta t^2)$. Throughout the proof we assume the existence of Fr\'{e}chet derivatives.
Writing, for any $j$, $u^j=u(\cdot,t^j)$ for the exact local solution at $t=t^j$, we get for the
left hand side 
\begin{multline} \label{eq:ordlhs}
\frac{u^{n+\barp}-u^n}{\barp\Delta t}=\partial_t u^n + \frac{\barp\,\Delta t}{2}\partial_t^2 u^n +\mathcal{O}(\Delta t^2)
= \left.\mathcal{D}
\frac{\partial\mathcal H}{\partial u}\right|_{u^n}\\
+\frac{\barp\Delta t}{2}\left.\left(\frac{\partial\cal D}{\partial u}(\partial_t u^n)
\frac{\partial\cal H}{\partial u}+\mathcal{D}\frac{\partial^2\cal H}{\partial u^2}
(\cdot,\partial_t u^n)\right)\right|_{u^n} + \mathcal{O}(\Delta t^2).
\end{multline}
Next we expand  \eqref{eq:tja} to get
\begin{multline*}
\frac{\delta \havf }{\delta(u^n,\ldots,u^{n+\barp})} =
\int_0^1 \frac{\delta H}{\delta w_1}(\xi u^{n+\barp}+(1-\xi)u^n,
u^{n+1},\ldots,u^{n+\barp-1})\,\mathrm{d}\xi \\
=\left.\frac{\partial H}{\partial w_1}\right|_{\bf u}+
\left(\frac{k}{2}\left.\frac{\partial^2 H}{\partial w_1^2}\right|_{\bf u}
+\Delta t\sum_{j=2}^\barp (j-1)\left.\frac{\partial^2 H}{\partial w_1\partial w_j}\right|_{\bf u}
\right) (\cdot,\partial_t u^n) +\mathcal{O}(\Delta t^2)
\end{multline*}
where $\mathbf{u}=(u^n,\ldots,u^n)$.
Using first \eqref{eq:2dpolar} and then \eqref{eq:firstder}, \eqref{eq:secder} we find
\begin{equation} \label{eq:avfexpanded}
\frac{\delta \havf }{\delta(u^n,\ldots,u^{n+\barp})} =
\frac{1}{\barp}\left.\frac{\partial\mathcal{H}}{\partial u}\right|_{u^n}+
\frac{\Delta t}{2}\left.\frac{\partial^2\mathcal{H}}{\partial u^2}\right|_{u^n} (\cdot,\partial_t u^n)+\mathcal{O}(\Delta t^2).
\end{equation}
Expanding $D$ we get
\begin{equation} \label{eq:Dexpand}
D[u^{n+1},\ldots,u^{n+\barp-1}]=\mathcal{D}[u^n] + 
\Delta t\sum_{j=1}^{\barp-1} j \left.\frac{\partial D}{\partial w_j}\right|_{\mathbf{u}}(\partial_tu^n)+\mathcal{O}(\Delta t^2).
\end{equation}
If the cyclicity condition \eqref{eq:Dcyclic} holds for $D$, we can simplify \eqref{eq:Dexpand} 
to obtain
\begin{multline} \label{eq:Dexpanded}
D[u^{n+1},\ldots,u^{n+\barp-1}]=\mathcal{D}[u^n] + \frac{\barp(\barp-1)\Delta t}{2}\left.\frac{\partial D}{\partial w_1}\right|_{\bf u}(\partial_t u^n)+ \mathcal{O}(\Delta t^2)\\
= \mathcal{D}[u^n] + \frac{\barp\Delta t}{2}
\left.\frac{\partial\cal D}{\partial u}\right|_{u^n}(\partial_t u_n)+\mathcal{O}(\Delta t^2).
\end{multline}
By substituting into \eqref{eq:scheme} the expressions \eqref{eq:ordlhs}, \eqref{eq:avfexpanded} and \eqref{eq:Dexpanded}, all terms of zeroth and first order cancel and we are left with $\mathcal{O}\left((\Delta t^2)\right)$. 
\end{proof}

	 	\begin{theorem}\label{th:linimp}
	Suppose that the polarised Hamiltonian $H$ is a quadratic polynomial in each of its arguments, then the  PAVF scheme  
         is linearly implicit  
 
	\end{theorem}
	\begin{proof}
		Since $H$ is at most quadratic in the first argument, it follows from
		\eqref{eq:euleropdef} that $\frac{\delta H}{\delta w_1}$ is of degree at most 1 in its first argument, 
		and so we see from \eqref{eq:tja} that
\[
   \frac{\delta H}{\delta(U^n,\ldots,U^{n+\barp})}
\]
is linear in $U^{n+\barp}$. 
Since $D$ does not depend on $U^{n+\barp}$ we conclude that the scheme  \eqref{eq:scheme} is linearly implicit.
	\end{proof}
	
	In some cases one wishes to have time-symmetric numerical schemes, see for example~\cite{MR2221614}. The numerical
	scheme~\eqref{eq:scheme} will in general not be symmetric, however it is not hard to modify the procedure to yield
	symmetric schemes. One needs to polarise $\mathcal{H}$ such that $H$ is 
	invariant also when the order of its arguments is reversed, it turns out that this can be achieved by symmetrising over the dihedral group rather than just the cyclic one.
 A similar adjustment must be made for
	$D$. 

		We remark that one can construct explicit schemes by using $p$ time steps (as opposed to $\barp$) in $H$
		such that $H$ becomes $p$-linear (as opposed to $\barp$-quadratic).
		The rest of the procedure for the explicit case is the same as for the linearly implicit case. 
		Clearly, one expects that explicit schemes will have more severe stability restrictions 
		 than the linearly implicit ones. 
		
		Since these multistep schemes need the $k$ previous values, it is
		not self-starting. We have to provide the starting-values $U^{1},\dots,U^{k}$ in addition to the initial value $U^0$. Usually these are computed using another sufficiently accurate conservative scheme, such as for example the AVF scheme. Another possibility is to use any integrator and integrate to machine precision.

\subsection{Stability}\label{se:stability}
	In~\cite{dahlbyowren} we studied linearly implicit schemes for the cubic Schr\"{o}dinger equation, and found that two-step schemes can develop a two-periodic instability in time. 
	We also saw that this can be remedied by choosing a different polarisation of the Hamiltonian.
	
	As it turns out, a common case is when the Hamiltonian is a univariate polynomial
	of degree 4 or less. If we polarise this Hamiltonian using two time-steps, we get three linearly independent
	$H$, corresponding to~\eqref{eq:hh1}, \eqref{eq:hh2}, and~\eqref{eq:hh3}.
	The third and fourth degree Hamiltonians are uniquely given.
	 However, in the second degree case we can choose $\theta\in[0,1]$
	 such that the scheme becomes unstable.  
	Since Hamiltonians of the type~\eqref{eq:hh1} appear in many important PDEs
	it may be useful to 
	determine which $\theta\in[0,1]$ lead to unstable schemes.
	
	We choose to study the test equation with Hamiltonian
	\begin{equation*}
		\mathcal{H}[u]=\frac12\int_\Omega u_x^2\,\mathrm{d}x,
	\end{equation*}
	and a skew-symmetric operator $\mathcal{D}$ which satisfies the eigenvalue equation
	\begin{equation*}
		\mathcal{D}\mathrm{e}^{\mathrm{i}kx}=\mathrm{i}\lambda_k\mathrm{e}^{\mathrm{i}kx}, \quad \lambda_k\in\mathbb{R}
	\end{equation*}
	for all integers $k$. The Airy equation
	\begin{equation*}
		u_t+u_{xxx}=0
	\end{equation*}
	 is of this type with $\mathcal{D}=-\partial_x$ and $\lambda_k=-k$. 
	Other equations which have such terms in the
	Hamiltonian include the nonlinear Schr\"{o}dinger equation, the linear wave equation, the KdV equation, and the Kadomtsev-Petviashvili equation. 
	
	Rewriting~\eqref{eq:hh1} gives
	\begin{equation*}
		H[v,u]=\frac12\int\left(\theta\frac{u_x^2+v_x^2}2+(1-\theta)u_{x}v_{x}\right) \,\mathrm{d}x.
	\end{equation*}
	And the numerical scheme is
	\begin{equation}\label{eq:numairy}
		\frac{U^{n+2}-U^{n}}{2\Delta t}=-\mathcal{D}\left(\theta\frac{U_{xx}^{n+2}+U_{xx}^{n}}2+(1-\theta)U_{xx}^{n+1}\right).
	\end{equation}
	Since this is a linear equation we can use von Neumann stability analysis~\cite{MR0042799}. 
	We insert the ansatz
	\begin{equation*}
		U^n(x)=\zeta^n\mathrm{e}^{\mathrm{i}kx}
	\end{equation*}
	to obtain the quadratic equation
	\begin{equation}\label{eq:stabpol}
		(1-\theta\tau\mathrm{i})\zeta^2-2(1-\theta)\tau\mathrm{i}\zeta-(1+\theta\tau\mathrm{i})=0,\quad\tau=\lambda_k\Delta t k^2.
	\end{equation}	
	A necessary condition for stability is $|\zeta|\leq1$ which implies
	\begin{equation*}
		\theta\geq\frac12-\frac1{2\tau^2}.
	\end{equation*}
	Assuming that $\{\lambda_k k^2\}_{k\in\mathbb{Z}}$ is unbounded, we must require that 
	$\theta$ is chosen greater than or equal $\frac12$.  
	This is exactly the condition found in~\cite{dahlbyowren} for the cubic Schr\"odinger equation.  When $\theta\geq\frac12$
	the roots of \eqref{eq:stabpol} satisfy $|\zeta_1|=|\zeta_2|=1$.
	
	In Figure~\ref{fi:airystab} we solve the Airy equation with the scheme~\eqref{eq:numairy} using $\theta=0.5$ and $\theta=0.49$.
	We use the initial value $u(x,0)=\sin(x)$, which, in the exact case, yields the traveling wave solution $u(x,t)=\sin(x+t)$.
	The  $\theta=0.49$ solution blows up in few time steps, while the $\theta=0.50$ solution shows no signs of instability.
	Doing a discrete Fourier transform of the unstable solution we see that the instability starts at high frequencies, that is large $k$, which
	corresponds to the results shown above. 
	\begin{figure}
		\centering
		\includegraphics[width=0.8\textwidth]{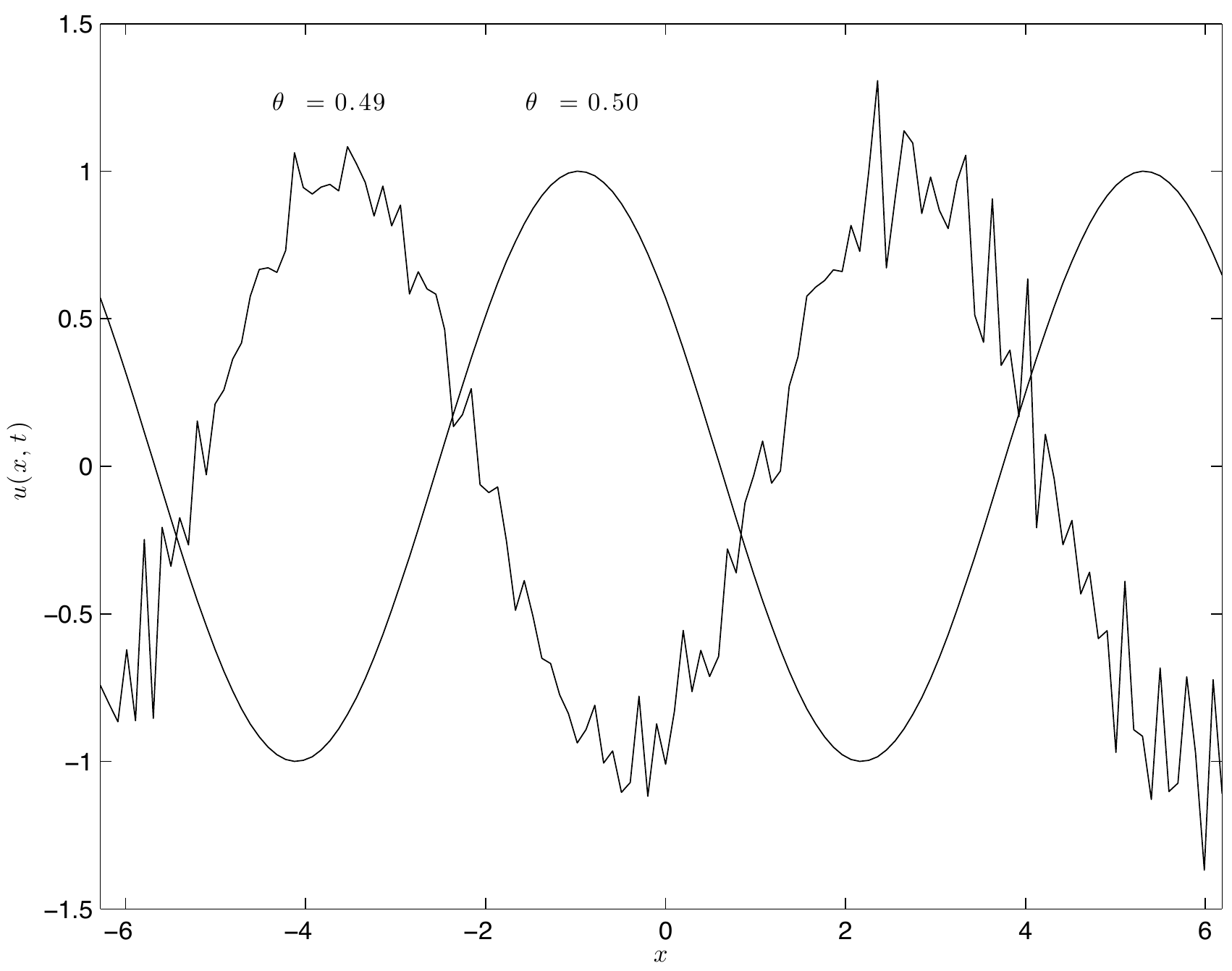}
		\caption{The numerical solution of the Airy equation with two different values of $\theta$. The two solutions are shown after $n=10^6$ time steps ($\theta=0.5$) and $n=115$ time steps ($\theta=0.49$). }
		\label{fi:airystab}
	\end{figure}
	
	There might be cases where the scheme develop instabilities due to spurious modes no matter what polarisation one chooses. A full stability analysis of either the fully or linearly implicit schemes is to our knowledge not been done. Standard linearisation techniques will usually lead to the conclusion that the schemes are neutrally stable. The nonlinear effects, however small, may still cause the scheme to be unstable. The tests we have done on a wide range of PDEs seem to indicate that the stability usually is very good, future work may shed a light on this issue.

\section{Space discretisation}\label{se:spatial}
	Until now we have mostly considered the situation where the PDE is discretised in time while remaining continuous in space. The methodology developed in the previous sections apply equally well to systems of ODEs.
	Arguably, the most straightforward approach is simply to discretise the space derivatives in the Hamiltonian, for instance by finite differences. This leads to
\[	
   \mathcal{H}(u) \longrightarrow \mathcal{H}_d(\mathbf{u}).
\]
One also needs to replace the skew-symmetric operator $\mathcal{D}$ by a skew-symmetric matrix $\mathcal{D}_d$.	
The fully implicit method reviewed in Section~\ref{sec:dg} is then just the discrete gradient method
\eqref{eq:metdg}, which conserves the discretised Hamiltonian $\mathcal{H}_d(\mathbf{u})$ in every time step.  

We consider now finite difference approximations.
 The function space to which the solution $u$ belongs, is replaced by a finite dimensional space with 
 functions on a grid indexed by $I_g\subset \mathbb{Z}^d$.
 We use boldface symbols for these functions. Let there be $N_r$ grid points in the space direction $r$ so that
$\mathbf{N}=N_1\cdots N_d$ is the total number of grid points. We denote by $\mathbf{u}^{\alpha}$ the approximation
to $u^{\alpha}$ on such a grid, and by $\mathbf{u}$ the vector consisting of $(\mathbf{u}^1,\ldots,\mathbf{u}^m)$.
We will
replace each derivative $u_J^{\alpha}$ by a finite difference approximation $\delta_J\mathbf{u}^{\alpha}$, and replace the integral by a quadrature rule. 
%
%
%
We then let
\begin{equation} \label{eq:discH}
\mathcal{H}_d(\mathbf{u}) =  \sum_{\mathbf{i}\in I_g} b_{\mathbf{i}} (\mathcal{G}_d(  (\delta_J {\bf u})))_{\mathbf{i}}\, \Delta x.
\end{equation}
Here $\Delta x$ is the volume (length, area) of a grid cell and ${\bf b}=(b_{\mathbf{i}})_{\mathbf{i}\in I_g}$ are the weights in the quadrature rule. The discretised $\mathcal{G}_d$ has the same number of arguments as $\mathcal{G}$, and each input argument as well as the output are vectors in $\mathbb{R}^N$. We have here approximated the function $u_J^\alpha$ by a difference approximation
$\delta_J \mathbf{u}^{\alpha}$, where $\delta_J:\mathbb{R}^{\bf N}\rightarrow\mathbb{R}^{\bf N}$ is a linear map. As in the continuous case, we use square brackets, say $F[\mathbf{u}]$, as shorthand for a list of arguments involving difference operators $F[\mathbf{u}]=F(\mathbf{u},\delta_{J_1}\mathbf{u},\ldots,\delta_{J_q}\mathbf{u})$.
We compute
\begin{multline}
\mathcal{H}_d[\mathbf{u}] -\mathcal{H}_d[\mathbf{v}] =\\
\sum_{{\bf i}\in I_g} b_{\bf i} \sum_{J,\alpha}\int_0^1
\left(\frac{\partial\mathcal{G}_d}{\partial\delta_J\mathbf{u}^{\alpha}}\right)_{\bf i}[
\xi\mathbf{u}
+(1-\xi)\mathbf{v}] {\rm d}\xi  
(\delta_J(\mathbf{u}^{\alpha}-\mathbf{v}^{\alpha}))\;\Delta x\\= 
  \langle \frac{\delta\mathcal{H}_d}{\delta(\mathbf{v},\mathbf{u})}, \mathbf{u}-\mathbf{v}\rangle
  \label{eq:discip}
\end{multline}
where
\begin{equation*}
\frac{\delta\mathcal{H}_d}{\delta(\mathbf{v},\mathbf{u})}
= \sum_{J,\alpha} \delta_J^T\,B\, \left(\int_0^1\frac{\partial\mathcal{G}_d}{\partial\mathbf{u}_J}
[\xi\mathbf{u}^{\alpha}+(1-\xi)\mathbf{v}^{\alpha}]\,\mathrm{d}\xi\right),
\end{equation*}
$B$ is the diagonal linear map $B=\mathrm{diag}(b_{\mathbf{i}}),\ \mathbf{i}\in I_g$, and
the discrete inner product used in \eqref{eq:discip} is
\[
   \langle \mathbf{u}, \mathbf{v} \rangle = \sum_{\alpha,\mathbf{i}\in I_g} \mathbf{u}^{\alpha}_{\mathbf{i}} \mathbf{v}^{\alpha}_{\mathbf{i}}.
\]
Notice the resemblance between the operator acting on $\mathcal{G}_d$ in \eqref{eq:discip} and the continuous Euler operator in \eqref{eq:euleropdef}.  Alternatively, suppose that 
\begin{enumerate}
\item The spatially continuous method \eqref{eq:dgm} (using \eqref{eq:dgavf}) is discretised in space, using 
a skew-symmetric $\mathcal{D}_d$ and a selected set of difference quotients $\delta_J$ for each derivative $\partial_J$.
\item Considering \eqref{eq:varderdef} and \eqref{eq:euleropdef}, the choice of discretisation operators $\delta_J$ used in
$\partial\mathcal{G}/\partial u_J^{\alpha}[u]$ is arbitrary, but the corresponding $D_J$ is replaced by the transpose
$\delta_J^T$.
\end{enumerate}
In this case, using the same $\mathcal{D}_d$, an identical set of difference operators in discretising $\mathcal{H}$ \eqref{eq:discH}, and choosing all the quadrature weights $b_{\mathbf{i}}=1$ the resulting scheme would be the same as that given by procedure outlined in the two points above. That is, one can get the same scheme by either discretising the Hamiltonian in space first (and then deriving the scheme) or discretising the scheme in space first (and then deriving the conserved Hamiltonian). 

Letting the $r$th canonical unit vector in $\mathbb{R}^d$ be denoted $\mathbf{e}_r$, we define the most used first order difference operators
\begin{align*}
(\delta_{r}^+\mathbf{u})_{\mathbf{i}}&=\frac{\mathbf{u}_{\mathbf{i}+\mathbf{e}_r}-\mathbf{u}_{\mathbf{i}}}{\Delta x_r}, \\
(\delta_{r}^-\mathbf{u})_{\mathbf{i}}&=\frac{\mathbf{u}_{\mathbf{i}}-\mathbf{u}_{\mathbf{i}-\mathbf{e}_r}}{\Delta x_r}, \\
(\delta_{r}^{\langle 1\rangle}\mathbf{u})_{\mathbf{i}}&=\frac{\mathbf{u}_{\mathbf{i}+\mathbf{e}_r}
- \mathbf{u}_{\mathbf{i}-\mathbf{e}_r}}{2\Delta x_r}.
\end{align*}
These difference operators are all commuting, but only the last one is skew-symmetric. However, for the first two one has the useful identities
\begin{equation*}
    (\delta_{r}^+)^T = -\delta_{r}^ -,\qquad  (\delta_{r}^-)^T = -\delta_{r}^ +.
\end{equation*}
Higher order difference operators $\delta_J$ can generally be defined by taking compositions of these operators, in particular we shall consider examples in the next section using the second and third derivative approximations
\[
  \delta_r^{\langle 2\rangle} = \delta_r^+\circ\delta_r^-,\quad \delta^{\langle 3\rangle}=
   \delta^{\langle 1\rangle}\circ
    \delta^{\langle 2\rangle}.
\]
We may now introduce numerical approximations $\mathbf{U}^n$ representing the fully discretised system, the scheme is
\[
\frac{\mathbf{U}^{n+1}-\mathbf{U}^n}{\Delta t} = \mathcal{D}_d\frac{\delta\mathcal{H}_d}{\delta(\mathbf{U}^n,\mathbf{U}^{n+1})}.
\]

The conservative schemes based on polarisation are adapted in a straightforward manner, introducing a function
$H_d[\bw_1,\ldots,\bw_{\barp}]$ which is consistent and cyclic as in \eqref{eq:consist}, \eqref{eq:cyclic}, and
a skew-symmetric map $D_d$ depending on at most $k-1$ arguments. 
The scheme is then
	\begin{equation}\label{eq:discscheme}				
		\frac{\mathbf{U}^{n+\barp}-\mathbf{U}^{n}}{\barp\Delta t}=\barp D_d\frac{\delta H_d}
		{\delta(\mathbf{U}^{n},\dots,\mathbf{U}^{n+\barp})}.
	\end{equation}		
	This scheme conserves the function $H_d$ in the sense that
\[
    H_d[\mathbf{U}^{n+1},\ldots,\mathbf{U}^{n+\barp}]= H_d[\mathbf{U}^{0},\ldots,\mathbf{U}^{\barp-1}],\quad n\geq 0.
\]

\section{Examples}\label{se:num}
	To illustrate the procedures for constructing conservative schemes presented in  this paper we consider as an example the generalised Korteweg-de Vries (gKdV) equation
	\begin{equation*}
		u_t+u_{xxx}+(u^{p-1})_x=0
	\end{equation*}
	for an integer $p\geq3$, see for example~\cite{MR2457070}. 
	The case $p=3$ is the KdV equation \eqref{eq:KdV}, the case $p=4$ is known as the modified KdV equation, and
	$p=6$ is sometimes referred to as the mass critical generalised KdV equation. The gKdV can be written as \eqref{eq:hamileq} with
	\begin{equation*}
		\mathcal{H}[u]=\int_\Omega\left(\frac12u_x^2-\frac{1}{p}u^{p}\right)\,\mathrm{d}x,\quad
		\mathcal{D}=\frac{\partial}{\partial x}.
	\end{equation*}
	The AVF discrete variational derivative \eqref{eq:dgavf}
	gives rise to the fully implicit scheme \eqref{eq:dgm}
	\begin{equation}\label{eq:genkdvimp}
		\frac{U^{n+1}-U^n}{\Delta t}+\frac{U^{n+1}_{xxx}+U^{n}_{xxx}}2+\frac1{p}\left(\sum_{i=0}^{p-1}(U^{n+1})^{p-1-i}(U^{n})^i\right)_x=0.
	\end{equation}
	
	After applying the polarising procedure of Section \ref{se:pol} to $\mathcal{H}=\int_\Omega\mathcal{G}\,\mathrm{d}x$ we get $H=\int_\Omega G\,\mathrm{d}x$ which depends on  $k=\lceil p/2\rceil$ arguments
	\begin{equation*}
		G[w_1,\dots,w_k]=
		\frac1{2k}\sum_{i=1}^k(w_i)_x^2-
		\begin{cases}
			\frac1{pk}\left(\prod_{j=1}^kw_j^2\right)\left(\sum_{i=1}^k\frac{1}{w_i}\right), & p\text{ odd},\\
			\frac1{p}\prod_{j=1}^kw_j^2, & p\text{ even}.
		\end{cases}
	\end{equation*}
	After finding the 
	AVF discrete variational derivative from \eqref{eq:tja} we get the linearly implicit PAVF scheme \eqref{eq:scheme}
	\begin{multline}\label{eq:genkdvlinin}
		\frac{U^{n+k}-U^n}{k\Delta t}+\frac{U^{n+k}_{xxx}+U^{n}_{xxx}}2\\+
		\begin{cases}
			\frac1p\left[\left(\prod_{j=1}^{k-1}(U^{n+j})^2\right)\left(\sum_{i=1}^{k-1}\frac{U^{n+k}+U^{n}}{U^{n+i}}+1\right)\right]_x=0,& p\text{ odd},\\
			\frac12\left[\left(\prod_{j=1}^{k-1}(U^{n+j})^2\right)\left(U^{n+k}+U^{n}\right)\right]_x=0, & p\text{ even}.
		\end{cases}
	\end{multline}
	Notice that $U^{n+k}$ is indeed only appearing as linear terms in this scheme. 
	The schemes \eqref{eq:kdvimp} and \eqref{eq:kdvlinin} are found by setting $p=3$ ($k=2)$ in \eqref{eq:genkdvimp} and \eqref{eq:genkdvlinin}, respectively. Following the procedure of Section \ref{se:spatial} one can get a fully discretised scheme by replacing $U$ by $\mathbf{U}$ and the first and third derivative operators by $\delta^{\langle 1\rangle}$ and $\delta^{\langle 3\rangle}$ respectively.
	
	In the Figures \ref{fi:solitonerror} and \ref{fi:faseerror} we compare the conservative methods \eqref{eq:kdvimp} and \eqref{eq:kdvlinin} with the fully implicit midpoint method
	\begin{equation}\label{eq:fi}
		 \frac{U^{n+1}-U^{n}}{\Delta t}+\frac{U_{xxx}^{n+1}+U_{xxx}^{n}}2+\left(\left(\frac{U^{n+1}+U^{n}}2\right)^2\right)_x=0
	\end{equation}
	and a naive linearly implicit method
	\begin{equation}\label{eq:li}
		 \frac{U^{n+1}-U^{n}}{\Delta t}+\frac{U_{xxx}^{n+1}+U_{xxx}^{n}}2+\left(U^nU^{n+1}\right)_x=0.
	\end{equation}	
	We test the four methods on a traveling wave solution
	\begin{equation*}
		\Phi(x-ct)=\frac{3c}{2}\mathrm{sech}^2\left(\frac{3\sqrt{c}}{2}(x-ct)\right),\quad c>0
	\end{equation*}
	using the parameters $c=1$, $x=(-5,5)$, $\Delta x=\frac{10}{32}$ and $\Delta t=0.1$. 
	As an indication of the long time behaviour of the presented schemes we consider to which extent the methods are able to preserve the shape and propagation speed of a traveling wave solution. 
	  We define the two quantities 
	\begin{align}
		\varepsilon_{\mathrm{shape}}&=\min_{\tau} \lVert U^n-\Phi(\cdot-\tau)\rVert_2^2\label{eq:shape}
		\intertext{and}
		\varepsilon_{\mathrm{distance}}&=|\underset{\tau}{\mathrm{argmin}} \lVert U^n-\Phi(\cdot-\tau)\rVert_2^2-ct^n|.\label{eq:distance}
	\end{align}
	Thus $\varepsilon_{\mathrm{shape}}$ measures the shape error of the numerical solution, and $\varepsilon_{\mathrm{distance}}$ measures the error in the travelled distance of the numerical solution. 
	
	We see in Figure \ref{fi:solitonerror} the fully implicit schemes preserves the shape better than the linearly implicit ones, and that the conservative schemes perform better than the non-conservative ones. In Figure \ref{fi:faseerror} we see that the linearly implicit schemes have a more accurate phase speed than the fully implicit ones. 
	Figure \ref{fi:orderplot} shows the global error as a function of the time step. As expected the plot shows that the four methods are second order and that the linearly implicit schemes are inaccurate for large $\Delta t$.
	In conclusion we see that the linearly implicit conservative scheme performs comparably to the other methods while being more efficient (the latter is shown in Figure \ref{fi:secondpic}).
	\begin{figure}
		\centering
		\includegraphics[width=.8\textwidth]{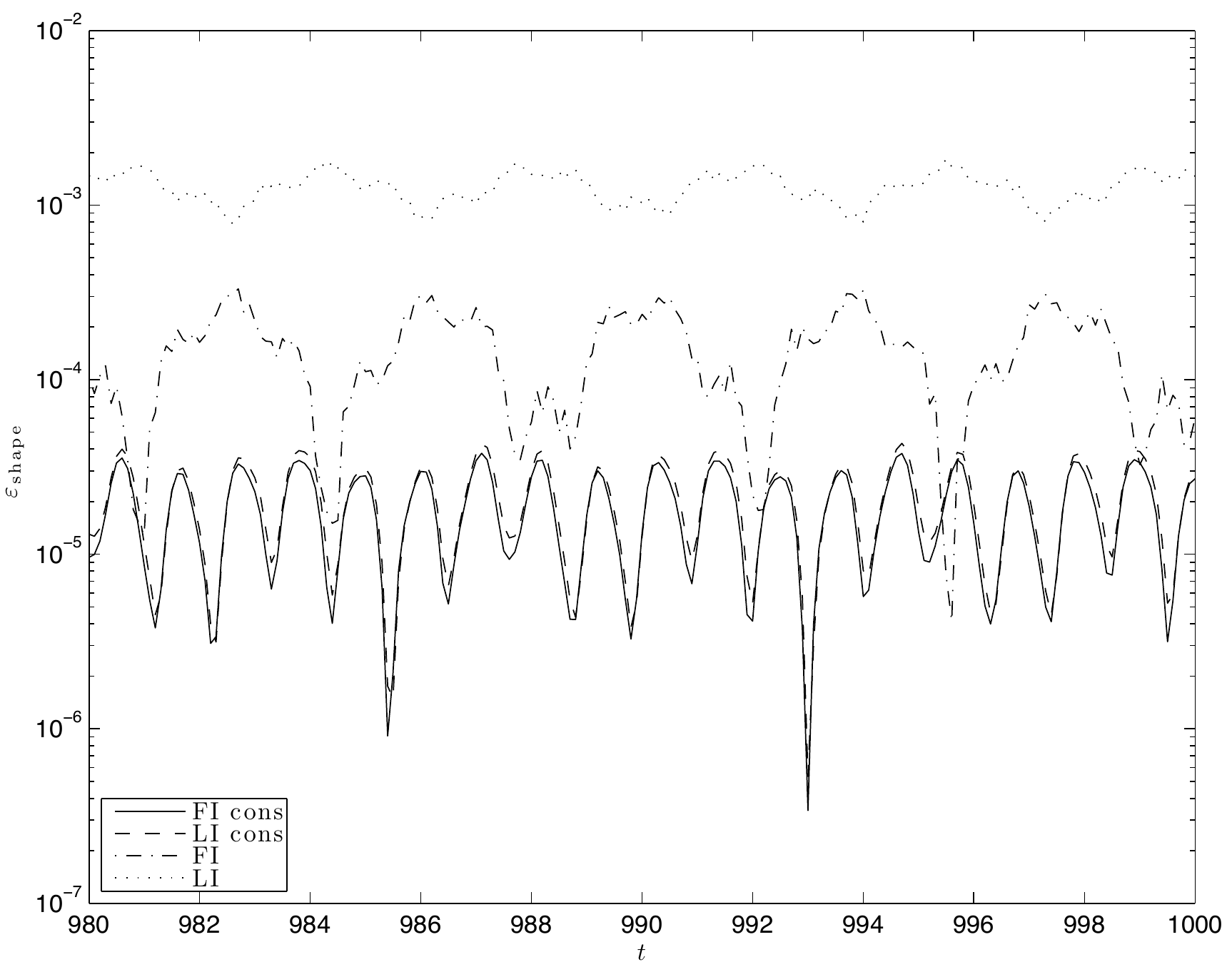}
		\caption{The shape error $\varepsilon_{\mathrm{shape}}$ \eqref{eq:shape} for the schemes \eqref{eq:kdvimp} (FI cons), \eqref{eq:kdvlinin} (LI cons), \eqref{eq:fi} (FI), and \eqref{eq:li} (LI). Only the largest times are shown, the plot is similar for smaller $t$. }
		\label{fi:solitonerror} 
	\end{figure}	
	\begin{figure}
		\centering
		\includegraphics[width=.8\textwidth]{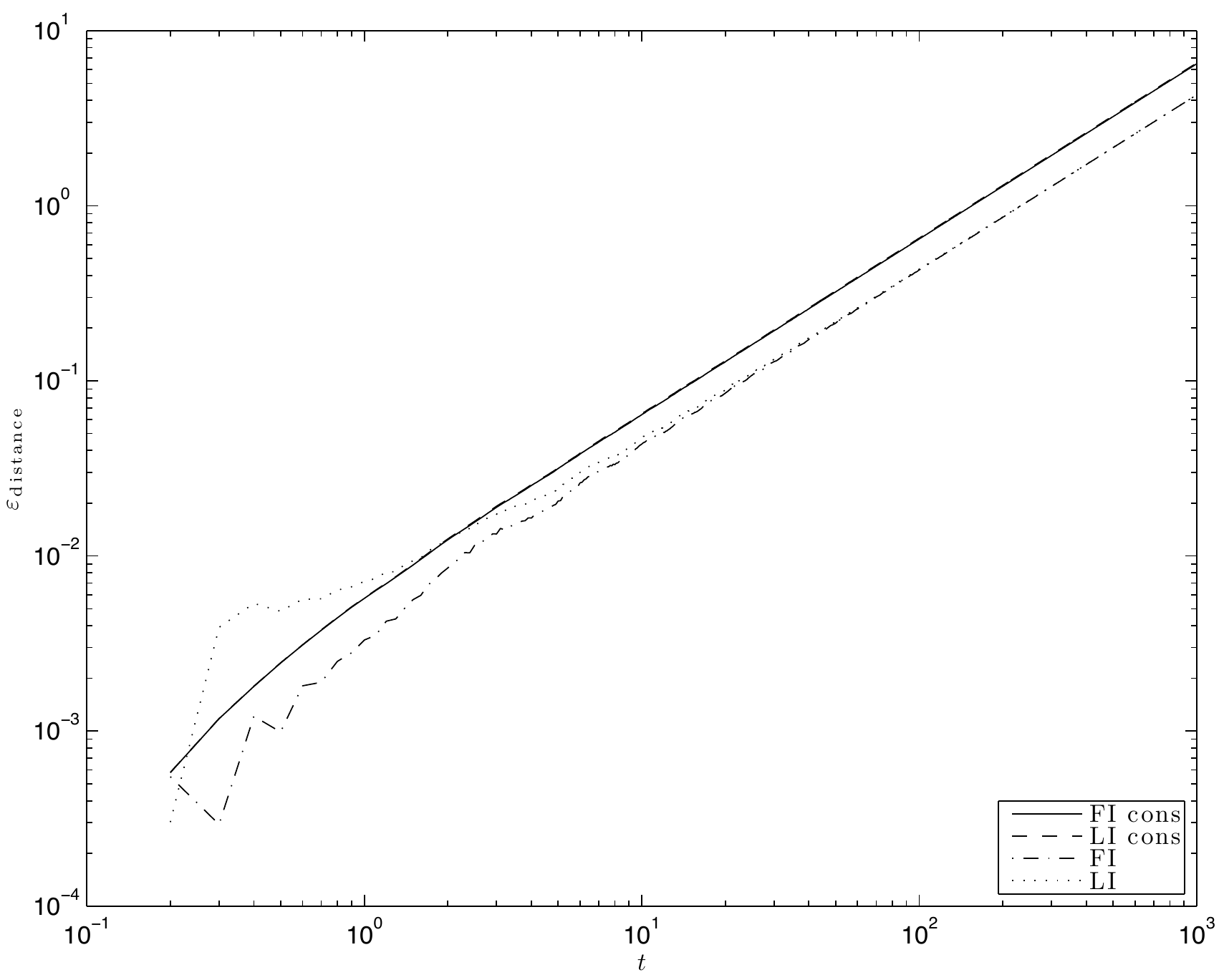}
		\caption{The distance error $\varepsilon_{\mathrm{distance}}$ \eqref{eq:distance} for the schemes \eqref{eq:kdvimp} (FI cons), \eqref{eq:kdvlinin} (LI cons), \eqref{eq:fi} (FI), and \eqref{eq:li} (LI). (FI cons) and (FI) are almost indistinguishable in this plot.  }
		\label{fi:faseerror} 
	\end{figure}	
	\begin{figure}
		\centering
		\includegraphics[width=.8\textwidth]{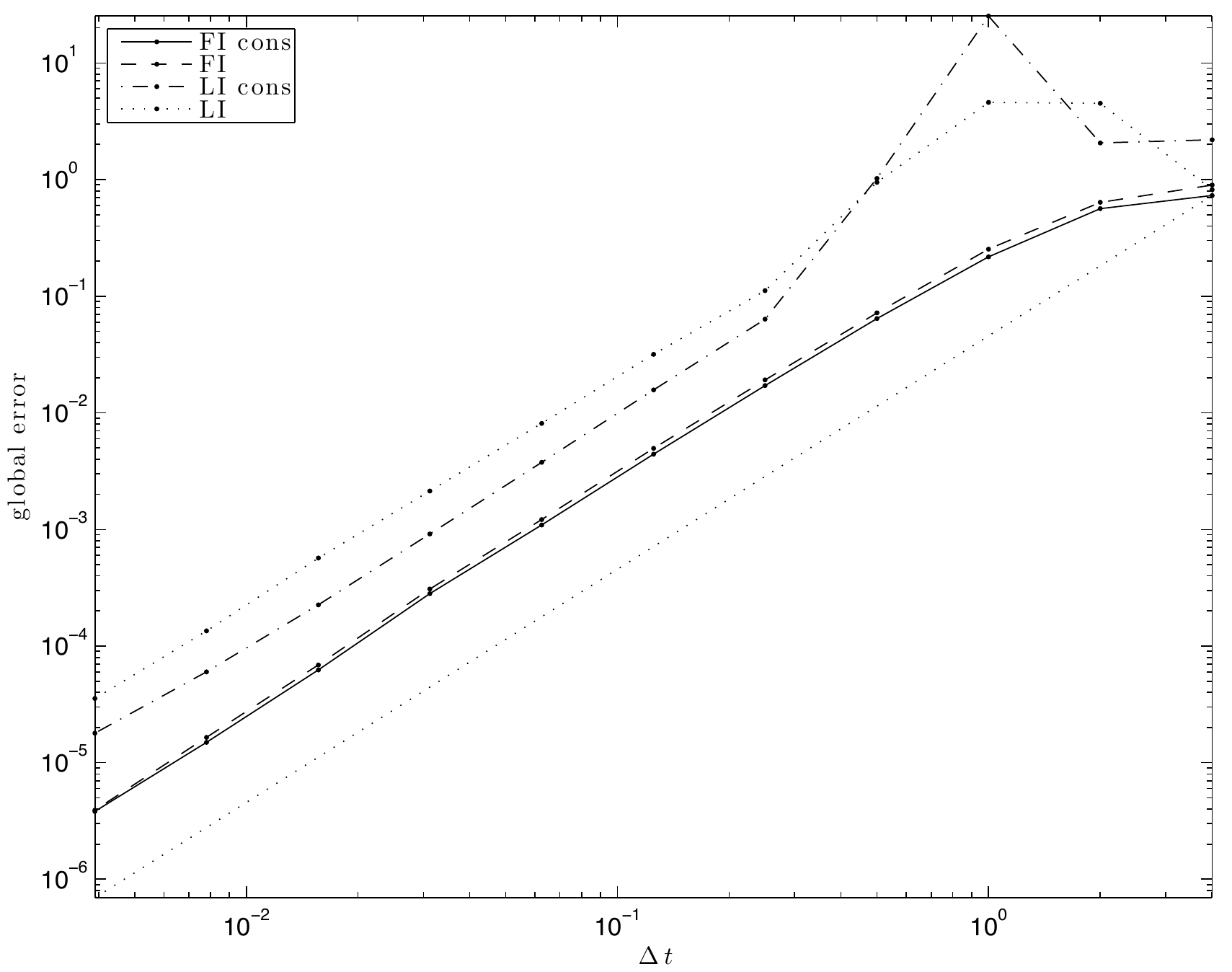}
		\caption{The global error at $t=8$ for the schemes \eqref{eq:kdvimp} (FI cons), \eqref{eq:kdvlinin} (LI cons), \eqref{eq:fi} (FI), and \eqref{eq:li} (LI). The dotted line is a second order reference line. }
		\label{fi:orderplot} 
	\end{figure}

 \bibliographystyle{plain}
 \bibliography{N8-2010}

\end{document}